\documentclass[12pt]{amsart}
\usepackage[T1]{fontenc}
\usepackage{amssymb}
\let\SavedRightarrow=\Rightarrow
\usepackage{marvosym}
\let\Rightarrow=\SavedRightarrow

\newcommand{\Bee }{\mathcal B}

\newcommand{\Pee }{\mathcal P}

\newcommand{\Tee }{\mathcal T}
\newcommand{\Xee }{\mathcal X}

\newcommand{\See }{\mathcal S}

\newcommand{\w}{\operatorname{w}}

\newcommand{\cl}{\operatorname{cl}}
\renewcommand{\int}{\operatorname{Int}}

\newtheorem{theorem}{Theorem}

\newtheorem{lemma}[theorem]{Lemma}

\author{Andrzej Kucharski}
\address{Andrzej Kucharski \\
 Institute of Mathematics, University of
Silesia \\
 ul. Bankowa 14, 40-007 Katowice}
\email{akuchar@ux2.math.us.edu.pl}
\author{Szymon Plewik}
\address{Szymon Plewik\\Institute of Mathematics,
University of Silesia, ul. Ban\-ko\-wa 14, 40-007 Katowice}
\email{plewik@math.us.edu.pl}

\begin{document}

\title{Skeletal maps and I-favorable  spaces} 
\subjclass[2000]{Primary: 54B10,  91A44; Secondary: 54C10, 91A05}
\keywords{Adequate pair, Inverse system; I-favorable space; Skeletal map}

\begin{abstract}
We show that the class of all compact Hausdorff and $I$-favorable  spaces is adequate for  the class of skeletal maps.  
\end{abstract}

\maketitle

 \section{Introduction} In \cite{s76}  E. V. Shchepin introduced so called adequate pairs. Suppose $\Xee$ is a class  
of compact Hausdorff spaces and $\Phi$ is a class  of continuous maps. The pair $(\Xee, \Phi)$ fulfills the condition 
of \textit{closure}, if the inverse limit of a continuous sequence $\{X_\alpha, p^\beta_\alpha , \Sigma \}$
belong to $\Xee$ and all  projections $\pi_\beta \in \Phi$, whenever 
all  spaces $X_\alpha\in \Xee$ of and all  bounding maps  $p^\beta_\alpha \in \Phi$. The pair $(\Xee, \Phi)$ fulfills the
 condition of \textit{decomposability}, whenever   every non-metrizable space
$X \in \Xee$ is the inverse limit of a continuous sequence $\{X_\alpha; p^\beta_\alpha; \alpha < \beta <\w(X)\}$, where 
 $\w(X_\alpha) <\w(X)$   and 
$X_\alpha \in \Xee$ and  $p^\beta_\alpha \in \Phi$. Finally, $(\Xee, \Phi)$ is an  \textit{adequate pair},  whenever it  fulfills 
  closure and decomposability.  
The aim of this note is to show that the class of compact Hausdorff   and $I$-favorable  spaces is adequate for  the class
 of skeletal maps.

  A directed set $\Sigma $ is said to be
$\sigma$-\textit{complete} if any countable chain of its elements has
least upper bound in $\Sigma$. 
 An inverse system $  \{ X_\sigma , \pi^\sigma_\varrho, \Sigma\}$ is said to be a $\sigma$-\textit{complete},
 whenever  $\Sigma $ is 
$\sigma$-complete and for every chain $\{\sigma_n: n \in \omega \} \subseteq \Sigma$, such that 
$\sigma = \sup \{\sigma_n: n \in \omega \} \in \Sigma,$ there  holds $$X_{\sigma }= \varprojlim \{ X_{\sigma_n}, \pi^{\sigma_{n+1}}_{\sigma_n}\}.$$ 
We consider inverse systems where bounding maps are surjections, only. Details about inverse systems 
one can find in \cite{eng}, pages 135 - 144.

A continuous surjection  is called \textit{skeletal},  whenever for  any non-empty open sets $U\subseteq X$ the
 closure of  $f[U]$ has non-empty interior. If $X$ is a compact space and  $Y$  Hausdorff, then a continuous 
surjection $f:X\to Y$ is skeletal if, and only if $\int f[U] \not=\emptyset,$  for every non-empty and open $U\subseteq X$.

It is well know - compare  a comment following the definition of compact open-generated spaces in \cite{s81} - 
that  each inverse  system  with open bounding maps has open limit projections. And conversely, if all limit 
projections of an inverse system are open,
then so are all bounding maps. 
The following fact is stated in \cite[Lemma 3]{bla}. Its proof is given in \cite[Proposition 8]{kp8}: 
\textit{If  $\{ X_\sigma, \pi^\sigma_\varrho,\Sigma\}$ is  an  inverse system such that  all bounding 
maps $\pi^\sigma_\varrho$ are skeletal and all projections  $\pi_\sigma$ are onto, then any projection 
$\pi_\sigma$ is skeletal}.

 \section{On $I$-favorable spaces}  
 Let  $X$ be a topologigal space equipped with a topology $\Tee$. The space $X$ is called I-\textit{favorable}, 
whenever there exists    a  function $$\sigma :\bigcup \{ \Tee^n: n\geq 0\} \to \Tee $$  such that
 for each sequence  $B_0, B_1, \ldots$ consisting of non-empty elements of $\Tee$ with  
$B_0 \subseteq \sigma(\emptyset)$ and   $B_{n+1} \subseteq \sigma((B_0, B_1, \ldots, B_n))$, for each $n\in \omega$,    
  the union $B_0 \cup B_1 \cup B_2  \cup  \ldots$  is dense in $X$. The function $\sigma$ is called  
a  \textit{winning strategy}. In fact, one can take a $\pi$-base (or a base) instead of a topology in the 
definition of a winning strategy.   The definition of $I$-favorable spaces  was introduced by P. Daniels, K. Kunen  and H.~Zhou \cite{dkz}.
The next the lemma one can  conclude from \cite[Theorem 4.1]{bjz}.  
\begin{lemma}\label{sk}   A  skeletal image of $I$-favorable space is a $I$-favorable space. \end{lemma}
\begin{proof}  Let $f:X\to Y$ be a  skeletal map.   Suppose  a  function 
$\sigma_X :\bigcup \{ \Tee^n_X: n\geq 0\} \to \Tee_X $ witnesses that $X$ is $I$-favorable. 
Put  $$\sigma_Y (\emptyset)= \int\cl f[\sigma_X (\emptyset)].$$ If $V_0 \subseteq \sigma_Y (\emptyset)$, 
then put $B_0=f^{-1}(V_0) \cap  \sigma_X (\emptyset).$  Suppose that non-empty sets 
$V_0 \subseteq \sigma_Y(\emptyset)$ and   $V_{n} \subseteq \sigma_Y((V_0, V_1, \ldots, V_{n-1}))$ are 
choosen and sets $B_0, B_1, \ldots , B_{n-1}$ are defined, also. 
Put $$B_n=  f^{-1}(V_n) \cap  \sigma_X (B_0, B_1, \ldots , B_{n-1})$$ and 
$$\sigma_Y((V_0, V_1, \ldots, V_{n}))= \int\cl f[\sigma_X (B_0, B_1, \ldots , B_{n}) ].$$ 
The function $\sigma_Y $ witnesses that $Y$ is $I$-favorable, since    $V_0 \cup V_1 \cup \ldots$   
contains a dense set $f[ B_0 \cup B_1 \cup \ldots]$. \end{proof}

\section{Decomposability} We shall prove that the class of  compact Hausdorff and $I$-favorable spaces and the class of skeletal maps fulfill decomposability. Recall notions and  facts which are stated  in \cite{kp8}.
If $\Pee$ is family of subsets of $X$, then $y\in [x]_{\Pee}$ denotes that, for each $V\in \Pee$, $x\in V$ if, and only if $y\in V$; and $X{/\Pee}$ is the family of all  classes $[x]_{\Pee}$; and $X{/\Pee}$ is equipped with the coarser topology containing all  sets $ \{ [x]_\Pee: x \in V \}$, where $V\in \Pee$. Suppose $X$ is a topological space and $\Pee$ consists of open subsets of $X$. If $\Pee$ is closed under finite intersection, then the map $q ( x) = [x]_{\Pee}$ is continuous, see \cite[Lemma 1]{kp8}. Let  $\Pee_{seq}$ be the family of all sets $W$ which  satisfy the  following condition:  There exist sequences  of sets $ \{ U_n: n \in \omega\} \subseteq \Pee$ and $ \{ V_n: n \in \omega\} \subseteq \Pee$ such that $U_k\subseteq (X\setminus V_k)  \subseteq U_{k+1}$, for any $k\in\omega$, and $\bigcup \{ U_n: n \in \omega\}=W$.  If a ring  $\Pee$  of open subsets of $X$ is  closed   under a winning strategy and  $\Pee\subseteq  \Pee_{seq}$, then $X/{\Pee}$  is a completely regular space and the map $q:X\to X/{\Pee}$ is  skeletal, see  \cite[Theorem 10]{kp8}. In fact, we shall improve and generalize the following:
If $X$ is a $I$-favorable compact  space, then  $ X = \varprojlim \{ X_\sigma, \pi^\sigma_\varrho,\Sigma\},$ where   $\{ X_\sigma, \pi^\sigma_\varrho,\Sigma\}$ is a  $\sigma$-complete inverse system, all spaces $X_\sigma$ are compact and metrizable, and all bonding maps $\pi^\sigma_\varrho$ are skeletal and onto, see   \cite[Theorem 12]{kp8}.

Let  $\{X_\alpha: \alpha \in \Sigma\} $ be a family of Hausdorff topological spaces, where $(\Sigma, < ) $ is an upward directed set.  Suppose that there are given continuous functions $p^\beta_\alpha: X_\beta \to X_\alpha$  such that
$p^\gamma_\alpha = p^\gamma_\beta \circ p^\beta_\alpha$ 
 whenever  $\alpha <\beta < \gamma$. Thus    $\See = \{X_\alpha; p^\beta_\alpha;\Sigma \}$  is the inverse system with  continuous bounding maps. 
\begin{lemma}\label{321}  Let $X$ be a  Hausdorff topological space and $\See = \{X_\alpha; p^\beta_\alpha;\Sigma \}$   an inverse system with  continuous bounding maps.
   If there exist maps $\pi_\beta: X \to X_\beta$ such that each $\pi_\beta = p_\beta^\alpha \circ \pi_\alpha$ is onto $X_\beta$ and for each two different points $x, y \in X$ there exists $\alpha\in \Sigma$ with $\pi_\alpha(x) \not=\pi_\alpha (y)$, then there exists a one-to-one continuous map $f:X \to \varprojlim \See$  onto a dense subspace of $\varprojlim \See$. 
      \end{lemma}
\begin{proof} For any $x\in X$, put  $f(x)= \{\pi_\alpha (x)\}$. The function  $f$ is a  required one,  compare  \cite{eng} or \cite[Theorem 11]{kp8}. 
\end{proof}

Additionally, one concludes that $X$ has to be homeomorphic with  $ \varprojlim \See$, assuming that $X$ is compact in the above lemma. Now, we apply Lemma \ref{321} to an inverse sequence with a directed set $\Sigma$ which consists of infinite ordinal numbers less than the weight $\w(X)$ of a space $X$. 

\begin{theorem}\label{222} Any  a compact non-metrizable  and $I$-favorable space  $X $ is homeomorphic with the inverse limit of a continuous sequence $\{X_\alpha; p^\beta_\alpha; \omega \leq \alpha < \beta <\w(X) \}$, where each $X_\alpha$ is a compact Hausdorff and  $I$-favorable space,  with   $\w(X_\alpha )<\w(X)$,   and such that any  bounding map  $p^\beta_\alpha $ is skeletal.  
\end{theorem}

\begin{proof} For each cozero set $W\subseteq X$ fix a continuous function $f_W: X \to [0,1]$ such that $W= f_W^{-1}((0,1])$. Put  $\sigma_{2n} (W) = f_W^{-1}((\frac 1 n,1])$ and $\sigma_{2n+1} (W) = f_W^{-1}([0, \frac 1 n))$. Assume that $\sigma =\sigma_0$.
Fix a base $ \{V_\alpha: \alpha < \w(X)\}$   consisting of cozero sets of $X$.   If  $\omega \leq \beta < \w(X)$, then  $\Pee_\beta \supseteq  \{V_\alpha: \alpha < \beta\}$ is the least family  consisting  of cozero sets and  closed under finite unions and  under finite intersections  and  closed under  all functions $\sigma_n$. Thus $|\Pee_\beta|=|\beta|$ and  $\Pee_\gamma \subseteq \Pee_\beta$, whenever $\omega \leq \gamma \leq \beta$. Also, we get  $\Pee_\beta = \bigcup \{\Pee_\gamma: \gamma < \beta\}$ for a limit ordinal $\beta$.  Put $X_\beta = X{/\Pee_{\beta}}$ and $q_\beta (x) = [x]_{\Pee_\beta}$, thus    maps  $q^{\Pee_\beta}_{\Pee_\gamma} :  X_{\beta} \to X_\gamma$ are  skeletal. 

By Lemma \ref{321}, the inverse limit $$  \varprojlim \{ X_{\beta}; q^{\Pee_\beta}_{\Pee_\gamma}; \omega \leq \gamma < \beta <\w(X)\}$$ is homeomorphic to $X$ and each inverse limit   $$\varprojlim \{ X_{\beta}; q^{\Pee_\beta}_{\Pee_\gamma}; \omega \leq \gamma < \beta <\alpha\}$$ is  homeomorphic to  $X_\alpha$.
All spaces $X_\alpha$ are  skeletal images of  $X$, hence    they are $I$-favorable by Lemma \ref{sk}. We get  $\w(X_\beta) < \w(X)$, since  the family $\{ \{ [x]_{\Pee_\beta}: x \in V \}: V\in \Pee_\beta\}$ is a base for $X_\beta$, compare \cite[Lemma 1]{kp8}.
\end{proof}

\section{Closure}   
  We shall prove that the class of  compact Hausdorff and $I$-favorable spaces and the class of skeletal maps 
fulfill closure.  In fact, we shall  
 improve   \cite[Theorem 13]{kp8},  not assuming that spaces $X_\sigma$ have countable $\pi$-bases.

\begin{theorem}\label{33}
If $\See= \{ X_\alpha, \pi^\alpha_\varrho,\Sigma \}$ is a  $\sigma$-complete inverse system which consists of  
compact Hausdorff and $I$-favorable spaces  and  skeletal bounding maps $\pi^\alpha_\varrho$,  then the 
inverse limit $\varprojlim \See$ is a compact Hausdorff  and $I$-favorable space. 
\end{theorem}
\begin{proof} Let $\Bee$ be a base for $\varprojlim \See$ which consists of all sets 
 $\pi^{-1}_{\tau}(V)$, where $\tau\in\Sigma$ and each $V$ is an open subset of $X_{\tau}$. For each $\alpha \in \Sigma$, let $\sigma_\alpha$  be a winning strategy in the open-open game played on  $X_\alpha$. Fix $\tau_0 \in \Sigma$ and infinite and pairwise disjoint sets $A_n$ such that $\bigcup\{A_n: n\in\omega\}=\omega$ and  $n\in A_k$ implies $k \leq n$. Put  $ \sigma_\omega (\emptyset) = \pi^{-1}_{\tau_0} (\sigma_{\tau_0} (\emptyset))\subseteq \varprojlim \See$.
Suppose that $n \in A_k$ and all $\sigma_\omega ((B_0,B_1, \ldots , B_{n-1}))$ has been already defined such that   $\pi^{-1}_{\tau_m}(V_m)=B_m\subseteq \sigma_\omega ((B_0,B_1, \ldots , B_{m-1}))$ for  $0\leq m
\leq n$.  Thus indexes $\tau_0 < \tau_1 < \ldots < \tau_{n}$ are fixed.  
 If $n$ is the least element of $A_k$,  then $\tau_k\leq\tau_n$.  Put $$\sigma_\omega ((B_0,B_1, \ldots , B_n))=\pi_{\tau_k}^{-1}(\sigma_{\tau_k}(\emptyset)).$$ If  $\{i_0,i_1,\ldots , i_j\}=A_k\cap \{0,1,\ldots,n\}$ and $\tau_{i_0}<\tau_{i_1}<\ldots<\tau_{i_j}\leq\tau_n$, then let $$D_{i_0}=\int\pi_{\tau_k}(B_{i_0}), D_{i_1}=\int\pi_{\tau_k}(B_{i_1}),\ldots ,D_{i_j}=\int\pi_{\tau_k}(B_{i_j})\subseteq X_{\tau_k}$$ and   $$\sigma_\omega ((B_0,B_1, \ldots , B_n))=\pi_{\tau_k}^{-1}(\sigma_{\tau_k}((D_{i_0}, D_{i_1},\ldots ,D_{i_j})))\subseteq \varprojlim \See.$$ 
For other cases, put $\sigma_\omega((B_0,B_1, \ldots , B_{n}))\in \Bee$ arbitrarily. The strategy $ \sigma_\omega: \bigcup \{ \Bee^n: n\geq 0\} \to \Bee $ is just defined.  

Verify that $ \sigma_\omega$  is a winning strategy. Let   $\tau_0 < \tau_1 < \ldots $ and $B_0,
 B_1, \ldots$ be  sequences   such that  $\pi^{-1}_{\tau_0}(V_0)=B_0 \subseteq
 \sigma_\omega(\emptyset)$ and   $$\pi^{-1}_{\tau_{n+1}}(V_{n+1})=B_{n+1} \subseteq
 \sigma_\omega((B_0, B_1, \ldots, B_n)),$$  where  all $\tau_k\in\Sigma$
and each $B_k\in\Bee.$ If $\tau \in \Sigma$ is the least upper bound of $\{\tau_k: k\in
 \omega\}$, then always $$\pi^{-1}_{\tau}(\pi_{\tau}(B_k))=\pi^{-1}_{\tau}(\pi_{\tau
}(\pi^{-1}_{\tau_{k}}(V_{k})))=\pi^{-1}_{\tau}((\pi^{\tau}_{\tau_{k}})^{-1}(V_{k})))=
B_k.$$  Take an arbitrary base set  $(\pi^{\tau}_{\tau_{k}})^{-1}(W)\subseteq X_{\tau}$,
 where $W$ is an open subset in $X_{\tau_k}$.  Such sets consist of a base for $X_{\tau}$, since the inverse system is $\sigma$-complete and \cite[2.5.5. Proposition]{eng}. If $\sigma_{\tau_k}$ is a winning 
strategy on $X_{\tau_k}$,  then there exists $j\in A_k$ such that  $W\cap\int\pi_{\tau_k}(B_j)
\not =\emptyset$. Hence $$(\pi^\tau_{\tau_k})^{-1}(W)\cap \pi_\tau(B_j)\not =\emptyset,$$
Indeed, suppose that $(\pi^\tau_{\tau_k})^{-1}(W)\cap \pi_\tau(B_j)=\emptyset.$ Then $$\emptyset=\pi^\tau_{\tau_k}[(\pi^\tau_{\tau_k})^{-1}(W)\cap \pi_\tau(B_j)]=W\cap\pi^\tau_{\tau_k}[\pi_{\tau}(B_j)]\supseteq W\cap\int\pi_{\tau_k}(B_j),$$ a contradiction.
 Thus,  the union $\pi_\tau(B_0) \cup \pi_\tau(B_1) \cup \pi_\tau(B_2)  \cup  \ldots$  is dense in $X_\tau$. But $\pi_\tau$ is a skeletal map, hence $$\pi_\tau^{-1}(\pi_\tau(B_0) \cup \pi_\tau(B_1) \cup \pi_\tau(B_2)\cup  \ldots)=B_0 \cup B_1 \cup B_2  \cup  \ldots$$  has to be dense in $\varprojlim \See$.  
\end{proof}
  The class of compact Hausdorff   and $I$-favorable  spaces and the class of skeletal maps fulfills closure, since every continuous inverse sequence is a  $\sigma$-complete inverse system, too.   An equivalent version of Theorem \ref{33} for continuous inverse sequence, (using  Boolean algebras notions)  one can find in \cite[p. 197]{bjz}.

\end{document}